\documentclass[11pt]{amsart}

\usepackage[T1]{fontenc}
\usepackage[margin=1.1in]{geometry}
\usepackage{amsmath,amssymb,amsthm,mathtools}
\usepackage[expansion=false]{microtype}
\usepackage{hyperref}
\usepackage{enumitem}

\hypersetup{
  colorlinks=true,
  linkcolor=blue,
  citecolor=blue,
  urlcolor=blue
}

\newtheorem{theorem}{Theorem}[section]
\newtheorem{proposition}[theorem]{Proposition}
\newtheorem{lemma}[theorem]{Lemma}
\newtheorem{corollary}[theorem]{Corollary}
\theoremstyle{definition}

\newtheorem{remark}[theorem]{Remark}

\newcommand{\A}{\mathbb A}
\newcommand{\F}{\mathbb F}
\newcommand{\Z}{\mathbb Z}
\newcommand{\Q}{\mathbb Q}

\newcommand{\Gr}{\operatorname{Gr}}
\newcommand{\Hilb}{\operatorname{Hilb}}
\newcommand{\Hom}{\operatorname{Hom}}

\newcommand{\Sym}{\operatorname{Sym}}

\newcommand{\rank}{\operatorname{rank}}
\newcommand{\initt}{\operatorname{in}}
\newcommand{\m}{\mathfrak m}
\newcommand{\Rsm}{\mathcal R}
\newcommand{\Ttwo}{T_2}

\title[The one-step Shafarevich gap in embedding dimension five]
{The one-step Shafarevich gap in embedding dimension five}

\author{Chenyang Zhao}
\address{Department of Mathematics, Imperial College London, United Kingdom}
\email{cz2922@ic.ac.uk}

\date{}

\begin{document}

\begin{abstract}
Let $k$ be an algebraically closed field of characteristic zero and let
$S=k[x_1,\ldots,x_5]$ with maximal ideal $\m=(x_1,\ldots,x_5)$.  For a
codimension-$r$ subspace $Q\subset S_2$, set
\[
  I_Q=(Q)+\m^3.
\]
Then $S/I_Q$ has Hilbert function $(1,5,r)$.  We prove that the translated
one-step locus defined by these ideals is contained in the smoothable
component for every
\[
  r\in\{6,7,\ldots,15\}.
\]
We introduce a finite field differential rank certificate proving
dominance, for $6\le r\le 14$, of the Erman--Velasco map
\[
  \operatorname{GL}_5\times (\A^5)^r\dashrightarrow
  \Gr(r,\Sym^2 k^5),
  \qquad
  (g,a^{(1)},\ldots,a^{(r)})\mapsto
  g\cdot\langle q(a^{(1)}),\ldots,q(a^{(r)})\rangle,
\]
where
\[
  q(a)=\sum_{i=1}^5 a_i y_i^2-
  \left(\sum_{i=1}^5 a_i y_i\right)^2.
\]
The endpoint $r=15$ is handled separately by a flat degeneration of $21$ general
reduced points to the fat point defined by $\m^3$.  Combined
with the known small cases and with the known elementary components for
$r=3$ and $r=5$, this gives the complete one-step classification in embedding
dimension five: the one-step loci with Hilbert function $(1,5,r)$ are
smoothable for all $r\neq 3,5$, and the cases $r=3,5$ are precisely the generically reduced
elementary component cases.  In this sense the one-step Shafarevich gap in
embedding dimension five is completely resolved.
\end{abstract}

\maketitle

\section{Introduction}

Let $k$ be an algebraically closed field and let
\[
  S=k[x_1,\ldots,x_n], \qquad \m=(x_1,\ldots,x_n).
\]
For a codimension-$r$ subspace $Q\subset S_2$, the ideal
\[
  I_Q=(Q)+\m^3
\]
defines a local algebra
\[
  S/I_Q \simeq k\oplus S_1\oplus (S_2/Q)
\]
with Hilbert function $(1,n,r)$ and length $1+n+r$.  Translating the support
in $\A^n$ gives an irreducible family $V_{n,r}$ of dimension
\[
  \dim V_{n,r}=n+r\left(\binom{n+1}{2}-r\right).
\]
We call this the translated one-step locus.  The question whether $V_{n,r}$ is
an elementary component of the Hilbert scheme or instead lies in a larger
component is one of the classical problems around the work of Iarrobino
and Shafarevich on punctual Hilbert schemes and algebras with $\m^3=0$~\cite{Iarrobino,Shafarevich}.  Jelisiejew
states the precise dichotomy as the Shafarevich gap, Problem~XVI of
\cite{JelisiejewProblems}: for very compressed one-step ideals with
$\tfrac{(n-1)(n-2)}{6}+2<r\le\binom{n+1}{2}$, decide whether $V_{n,r}$ is a
generically reduced elementary component or is contained in the smoothable
component.  For $n=5$ this range is $5\le r\le15$.

This paper tackles the full one-step problem in embedding dimension five.  Put
\[
  N=\dim_k S_2=\binom{6}{2}=15.
\]
Throughout, $1\le r\le 15$, and $V_{5,r}\subset \Hilb^{6+r}(\A^5)$.  The small cases
are already understood.  The cases $r=1,2$ are smoothable by the small length
results of Cartwright--Erman--Velasco--Viray~\cite{CEVV}.  The case $r=3$ is
a generically reduced elementary component~\cite{Jelisiejew,GKS,GGGL}.
The case $r=4$ is smoothable.  In
particular, it is the exceptional embedding dimension five case in
Shafarevich's formula.  This is made explicit in the recent work of
Ga\l{}\k{a}zka--Keneshlou--\v{S}ivic~\cite[Proposition 4.2]{GKS}.  Finally,
Giovenzana--Giovenzana--Graffeo--Lella identify $(1,5,5)$ as a new
one-step elementary component and recall that $(1,5,4)$ is not such a
component~\cite[Remark 7.2]{GGGL}.

The remaining one-step cases in embedding dimension five therefore begin at
$r=6$.  Existing tangent space estimates rule out many of these loci as
\emph{generically reduced} elementary components, but such estimates do not
by themselves decide whether the loci are smoothable or lie in some other
non-smoothable component.  Thus the compatibility question is important:
our result should neither contradict the known elementary cases nor merely
reprove a tangent space obstruction.  The purpose of this paper is to prove
that the whole remaining tail lies on the smoothable side.  In this way the
result is a completion, for one-step loci in embedding dimension five, of the
case-by-case picture predicted by the existing literature.  Equivalently, it
settles the embedding dimension five instance of the one-step Shafarevich gap:
for every $1\le r\le15$, the locus $V_{5,r}$ is either one of the known
one-step elementary components or is contained in the smoothable component.

\begin{theorem}\label{thm:main-tail}
Let $k$ be an algebraically closed field of characteristic zero.  For every
\[
  6\le r\le 15,
\]
the translated one-step locus
\[
  V_{5,r}\subset \Hilb^{6+r}(\A^5)
\]
is contained in the smoothable component.
\end{theorem}

Combining Theorem~\ref{thm:main-tail} with the known cases gives the following
complete classification of the one-step loci in embedding dimension five.

\begin{corollary}\label{cor:classification}
Let $k$ be an algebraically closed field of characteristic zero.  Among the
translated one-step loci $V_{5,r}$, $1\le r\le 15$, the loci $V_{5,3}$ and
$V_{5,5}$ are the generically reduced elementary component cases.  For every
other $r$, the locus $V_{5,r}$ is contained in the smoothable component.
\end{corollary}

The new part of Theorem~\ref{thm:main-tail} is the dominance statement for
$6\le r\le 14$.  We use the smoothable regularity two ideals constructed by
Erman and Velasco~\cite{ErmanVelasco}.  In degree two, their inverse systems
are spanned by explicit quadrics
\[
  q(a)=\sum_{i=1}^5 a_i y_i^2-
  \left(\sum_{i=1}^5 a_i y_i\right)^2.
\]
For $6\le r\le 14$, we prove that the $\operatorname{GL}_5$-saturation of the
spans
\[
  \langle q(a^{(1)}),\ldots,q(a^{(r)})\rangle
\]
is dense in the full Grassmannian $\Gr(r,\Sym^2 k^5)$.  The proof is a
finite field differential rank certificate over $\F_{32003}$.  The computation
is exact, deterministic, and included as an ancillary Macaulay2 script~\cite{Macaulay2}.

The endpoint $r=15$ is different: the Grassmannian $\Gr(15,\Sym^2 k^5)$ is a
point, so there is no dominance computation to make.  In this case the algebra
is $S/\m^3$, and we prove smoothability by degenerating $21$ general reduced
points in $\A^5$ to the origin.

\section{One-step loci and inverse systems}

From now on, unless otherwise stated, $k$ is algebraically closed of
characteristic zero and
\[
  S=k[x_1,\ldots,x_5],\qquad \m=(x_1,\ldots,x_5).
\]
Let
\[
  W=\langle y_1,\ldots,y_5\rangle_k,
  \qquad
  \Ttwo=\Sym^2 W.
\]
The degree two Macaulay pairing identifies codimension-$r$ subspaces
$Q\subset S_2$ with $r$-dimensional subspaces
\[
  U=Q^\perp\subset \Ttwo.
\]
Conversely, if $U\subset \Ttwo$ is an $r$-plane, then
\[
  I_U=(U^\perp)+\m^3
\]
has Hilbert function $(1,5,r)$.  Thus the homogeneous one-step locus supported
at the origin is parametrized by
\[
  \Gr(r,\Ttwo)=\Gr(r,15).
\]
The translated one-step locus $V_{5,r}$ is the closure of the translates of
this homogeneous locus in $\Hilb^{6+r}(\A^5)$.

Let
\[
  \Rsm^5_d\subset \Hilb^d(\A^5)
\]
denote the smoothable component, namely the closure of the locus of $d$
distinct reduced points.  It is a closed subset of the Hilbert scheme and is
invariant under automorphisms of $\A^5$, in particular under translations.

\begin{lemma}\label{lem:dense-smoothable}
Suppose that a dense open subset of $\Gr(r,\Ttwo)$ parametrizes smoothable
ideals $I_U=(U^\perp)+\m^3$.  Then
\[
  V_{5,r}\subset \Rsm^5_{6+r}.
\]
\end{lemma}

\begin{proof}
The set $\Rsm^5_{6+r}$ is closed in $\Hilb^{6+r}(\A^5)$.  Hence, if a dense
open subset of the homogeneous Grassmannian lies in $\Rsm^5_{6+r}$, then the
whole homogeneous one-step locus lies in $\Rsm^5_{6+r}$.  Since translations
preserve the property of being a flat limit of reduced points, all translates
also lie in $\Rsm^5_{6+r}$.  Taking the closure gives the claim.
\end{proof}

\section{The Erman--Velasco parametrization}

We recall the part of the construction of Erman--Velasco~\cite{ErmanVelasco}
that is needed here.  Let $d,e$ be positive integers with
$e\le \binom{d+1}{2}$ and let $T_2=\Sym^2\langle y_1,\ldots,y_d\rangle$.
For $a=(a_1,\ldots,a_d)\in \A^d$, set
\[
  q(a)=\sum_{i=1}^d a_i y_i^2-
  \left(\sum_{i=1}^d a_i y_i\right)^2.
\]
Erman and Velasco show that, on a nonempty open subset of $(\A^d)^e$, the
span
\[
  \langle q(a^{(1)}),\ldots,q(a^{(e)})\rangle\subset T_2
\]
is the degree two inverse system space of the homogeneous initial ideal of a
configuration of $1+d+e$ reduced points.  In particular, every $e$-plane
arising from this open set gives a smoothable one-step algebra.  This is the
content of \cite[Proposition 3.2]{ErmanVelasco}.  See also their differential
rank method in \cite[Lemma 3.4]{ErmanVelasco}.

In the present paper $d=5$.  For $1\le r\le 15$, define a rational map
\[
  \Phi_r:\operatorname{GL}_5\times (\A^5)^r\dashrightarrow \Gr(r,\Ttwo)
\]
by
\[
  (g,a^{(1)},\ldots,a^{(r)})
  \longmapsto
  g\cdot\langle q(a^{(1)}),\ldots,q(a^{(r)})\rangle,
\]
where the map is defined on the open set on which the displayed quadrics are
linearly independent.

\begin{lemma}\label{lem:dominance-implies-smoothable}
If $\Phi_r$ is dominant, then a dense open subset of $\Gr(r,\Ttwo)$
parametrizes smoothable one-step ideals.  Consequently
$V_{5,r}\subset \Rsm^5_{6+r}$.
\end{lemma}

\begin{proof}
By the Erman--Velasco construction recalled above, there is a nonempty open
subset $\Omega\subset (\A^5)^r$ such that the corresponding spans of the
quadrics $q(a^{(j)})$ give smoothable one-step ideals.  The source
$\operatorname{GL}_5\times (\A^5)^r$ is irreducible.  If $\Phi_r$ is dominant,
then the intersection of $\operatorname{GL}_5\times\Omega$ with the domain
of definition of $\Phi_r$ is still a nonempty open subset, and its image is
dense in $\Gr(r,\Ttwo)$.  Thus a dense open subset of $\Gr(r,\Ttwo)$
parametrizes smoothable one-step ideals.  The final assertion follows from
Lemma~\ref{lem:dense-smoothable}.
\end{proof}

Thus the main task is to prove dominance of $\Phi_r$ for $6\le r\le 14$.

\section{Differential rank certificates}

We prove dominance by an exact differential rank computation over
$\F_{32003}$.  The target Grassmannian has dimension
\[
  \dim \Gr(r,\Ttwo)=r(15-r).
\]
It is therefore enough to find one point at which the differential of
$\Phi_r$ has rank $r(15-r)$.

Let
\[
  U=\langle q_1,\ldots,q_r\rangle\subset \Ttwo,
  \qquad q_j=q(a^{(j)}).
\]
The tangent space of the Grassmannian at $U$ is
\[
  T_U\Gr(r,\Ttwo)=\Hom(U,\Ttwo/U).
\]
The differential of $\Phi_r$ at $(1,a^{(1)},\ldots,a^{(r)})$ has two types of
columns.

First, varying the point $a^{(j)}$ in the $\ell$-th coordinate sends
\[
  q_j \longmapsto
  \frac{\partial q(a^{(j)})}{\partial a^{(j)}_\ell}
  =y_\ell^2-2\left(\sum_{i=1}^5 a_i^{(j)}y_i\right)y_\ell
  \pmod U,
\]
and sends $q_i$ to $0$ for $i\neq j$.  Second, an infinitesimal matrix
$H\in\mathfrak{gl}_5$ sends
\[
  q_j\longmapsto H\cdot q_j \pmod U,
  \qquad j=1,\ldots,r.
\]
If $q_j$ is represented by a symmetric matrix $M_j$, this infinitesimal action
is represented by
\[
  M_j\longmapsto HM_j+M_jH^{\mathsf T}.
\]

We use the ordered basis
\[
\begin{split}
  \mathcal B=(&y_1^2,y_1y_2,y_1y_3,y_1y_4,y_1y_5,
  y_2^2,y_2y_3,y_2y_4,y_2y_5, \\
  &y_3^2,y_3y_4,y_3y_5,
  y_4^2,y_4y_5,y_5^2)
\end{split}
\]
of $\Ttwo$.  The following proposition is the computational input.

\begin{proposition}\label{prop:rank-certificates}
Let $A_{14}$ be the following $14\times 5$ matrix over $\F_{32003}$:
\[
A_{14}=\begin{pmatrix}
0&0&2&2&2\\
0&1&2&2&1\\
1&0&0&2&1\\
2&0&2&2&2\\
0&2&0&2&0\\
1&2&1&0&2\\
2&2&2&0&1\\
1&0&1&2&2\\
2&2&2&2&0\\
1&0&2&0&0\\
1&2&2&2&2\\
2&1&2&2&1\\
0&1&1&2&0\\
1&2&1&1&1
\end{pmatrix}.
\]
For $6\le r\le 14$, let $A_r$ be the submatrix consisting of the first $r$
rows of $A_{14}$.  At the point $(1,A_r)$, the quadrics
$q(a^{(1)}),\ldots,q(a^{(r)})$ are linearly independent and the differential
of $\Phi_r$ has rank $r(15-r)$.
\end{proposition}

\begin{proof}
All computations are exact over $\F_{32003}$.  For a row $a$ of $A_r$, form the
coefficient vector of
\[
  q(a)=\sum_i a_i y_i^2-\left(\sum_i a_i y_i\right)^2
\]
in the basis $\mathcal B$.  Let $U(A_r)$ be the $r\times 15$ matrix with these
coefficient vectors as rows.

Instead of choosing a complement to $U=\operatorname{rowspan}U(A_r)$, we compute
the rank in the quotient intrinsically.  Let $C$ be the matrix whose columns
are the $5r+25$ infinitesimal direction vectors described above, regarded as
elements of $\Ttwo^{\oplus r}$.  Let $G$ be the matrix whose columns are the
$r$ basis vectors of $U$ placed in each of the $r$ slots of $\Ttwo^{\oplus r}$.
Then the rank of the differential in
\[
  \Hom(U,\Ttwo/U)\simeq (\Ttwo/U)^r
\]
is
\[
  \rank[\,C\mid G\,]-\rank G.
\]
Since the rows of $U(A_r)$ are independent, $\rank G=r^2$.

The ancillary Macaulay2 script computes the following table:
\[
\begin{array}{c|c|c|c|c}
 r & \rank U(A_r) & \rank[\,C\mid G\,] & \rank G
   & \rank d\Phi_r \\
\hline
 6  & 6  & 90  & 36  & 54 \\
 7  & 7  & 105 & 49  & 56 \\
 8  & 8  & 120 & 64  & 56 \\
 9  & 9  & 135 & 81  & 54 \\
 10 & 10 & 150 & 100 & 50 \\
 11 & 11 & 165 & 121 & 44 \\
 12 & 12 & 180 & 144 & 36 \\
 13 & 13 & 195 & 169 & 26 \\
 14 & 14 & 210 & 196 & 14
\end{array}
\]
For each row, the final entry equals $r(15-r)=\dim\Gr(r,\Ttwo)$.  Hence the
differential has full rank at $(1,A_r)$ for every $6\le r\le 14$.
\end{proof}

\begin{remark}\label{rmk:verification}
The computation in Proposition~\ref{prop:rank-certificates} is not randomized.
The matrix $A_{14}$ is fixed, and the script performs exact linear algebra over
$\F_{32003}$.  The same script also reports the sanity checks
\[
  \rank d\Phi_3=33<36,
  \qquad
  \rank d\Phi_5=45<50
\]
for the first three and first five rows of $A_{14}$, respectively.  This is
consistent with the known elementary component cases $r=3$ and $r=5$, where
$V_{5,r}$ is not smoothable: the method correctly does not certify
smoothability there.  It is a check on the implementation rather than an
independent proof, since the rank at a single point cannot show that $\Phi_r$
fails to be dominant.
\end{remark}

\begin{lemma}\label{lem:lifting-dominance}
For every $6\le r\le 14$, the map
\[
  \Phi_r:\operatorname{GL}_5\times (\A^5)^r\dashrightarrow \Gr(r,\Ttwo)
\]
is dominant over any algebraically closed field of characteristic zero.
\end{lemma}

\begin{proof}
The entries of the matrices used in Proposition~\ref{prop:rank-certificates}
are defined over $\Z[1/2]$.  The only division by $2$ occurs when converting a
quadratic form in coefficient coordinates to its symmetric matrix, and $2$ is
a unit modulo $32003$.

Fix $r$.  Since $\rank U(A_r)=r$ over $\F_{32003}$, some $r\times r$ minor of
$U(A_r)$ is nonzero modulo $32003$, and therefore the same minor is nonzero
over $\Q$.  Hence the quadrics are linearly independent at the corresponding
characteristic zero point.  Similarly, the equality
\[
  \rank[\,C\mid G\,]=15r
\]
over $\F_{32003}$ means that some maximal minor of $[\,C\mid G\,]$ is nonzero
modulo $32003$.  The same minor is therefore nonzero over $\Q$.  Also
$\rank G=r^2$ over $\Q$.  Consequently
\[
  \rank_{\Q} d\Phi_r=15r-r^2=r(15-r).
\]
After base change to any algebraically closed field of characteristic zero,
the differential is surjective at this point.

Restrict $\Phi_r$ to its domain of definition.  The source is irreducible and
the target Grassmannian is irreducible and smooth of dimension $r(15-r)$.
The rank of the differential at a point is bounded above by the dimension of
the closure of the image.  Since this rank equals $r(15-r)$, the closure of the
image has the same dimension as the Grassmannian.  Hence the closure of the
image is the whole Grassmannian, and $\Phi_r$ is dominant.
\end{proof}

\begin{corollary}\label{cor:tail-smoothable}
For every $6\le r\le 14$,
\[
  V_{5,r}\subset \Rsm^5_{6+r}.
\]
\end{corollary}

\begin{proof}
This follows immediately from Lemma~\ref{lem:lifting-dominance} and
Lemma~\ref{lem:dominance-implies-smoothable}.
\end{proof}

\section{\texorpdfstring{The endpoint $r=15$}{The endpoint r=15}}

It remains to handle $r=15$.  Since $\dim_k S_2=15$, this is the maximal
codimension, and no one-step locus has $r>15$.  Here $Q=0$, so the homogeneous
one-step ideal is simply $\m^3$ and the local algebra is
\[
  S/\m^3,
  \qquad
  \dim_k S/\m^3=1+5+15=21.
\]

\begin{lemma}\label{lem:m3-smoothable}
The algebra $S/\m^3$ is smoothable.  Hence
\[
  V_{5,15}\subset \Rsm^5_{21}.
\]
\end{lemma}

\begin{proof}
Choose $21$ general reduced points
\[
  \Gamma=\{p_1,\ldots,p_{21}\}\subset \A^5.
\]
Since
\[
  \dim_k S_{\le 2}=1+5+15=21,
\]
we may and do choose $\Gamma$ so that the evaluation map
\[
  S_{\le 2}\longrightarrow k^{21},
  \qquad f\longmapsto (f(p_1),\ldots,f(p_{21}))
\]
is an isomorphism.  Such configurations exist by the usual interpolation
induction over an infinite field: having chosen fewer than $21$ points, one can
choose the next point outside the zero locus of a nonzero polynomial in the
remaining linear system.  Let $I_\Gamma$ be the ideal of $\Gamma$.  Scale the
configuration by $t$, obtaining a family $t\Gamma$ for $t\ne0$, and take its
flat limit as $t\to0$.  Write $J=\initt_{\deg}^{\mathrm{top}}(I_\Gamma)$ for
the ideal generated by the top forms of the elements of $I_\Gamma$, where the
top form of $f=f_0+f_1+\cdots+f_D$ (each $f_i$ homogeneous, $f_D\ne0$) is
$f_D$.  For any such $f$ the polynomial $t^D f(x/t)$ lies in the ideal of
$t\Gamma$ when $t\ne0$ and specializes to $f_D$ as $t\to0$, so $f_D$ lies in
the limit ideal $I_0$.  Hence $J\subseteq I_0$.

Passing to top forms does not change the colength, so
$\dim_k S/J=\dim_k S/I_\Gamma=21$.  The family $t\Gamma$ is flat, so
$\dim_k S/I_0=21$ as well.  Combined with $J\subseteq I_0$, this forces
$J=I_0$.

Because the evaluation map on $S_{\le 2}$ is injective, $I_\Gamma$ contains no
nonzero polynomial of degree at most $2$, so every nonzero element of
$I_\Gamma$ has top form of degree at least $3$.  Thus $J$ contains no nonzero
form of degree $0$, $1$, or $2$, while $\dim_k S/J=21$.  The monomials of
degrees $0$, $1$, and $2$ already span a space of dimension $21$, so they form
a basis of $S/J$ and every form of degree at least $3$ lies in $J$.  Therefore
\[
  I_0=J=\m^3.
\]
Thus $S/\m^3$ is a flat limit of $21$ reduced points.  Translations preserve
smoothability, giving the assertion for $V_{5,15}$.
\end{proof}

\begin{proof}[Proof of Theorem~\ref{thm:main-tail}]
The cases $6\le r\le 14$ are Corollary~\ref{cor:tail-smoothable}.  The case
$r=15$ is Lemma~\ref{lem:m3-smoothable}.
\end{proof}

\section{Consequences for the embedding dimension five gap}

We now place the result beside the previously known cases and the existing
tangent space predictions.  The theorem solves the one-step problem for
Hilbert function $(1,5,r)$ in embedding dimension five.  It does not address the
problem of classifying all components of all Hilbert schemes in embedding
dimension five.
The following table summarizes the one-step loci $V_{5,r}$.
\[
\begin{array}{c|c|c|l}
 r & \text{Hilbert function} & \text{length} & \text{status} \\
\hline
 1 & (1,5,1)  & 7  & \text{smoothable by small length results} \\
 2 & (1,5,2)  & 8  & \text{smoothable by small length results} \\
 3 & (1,5,3)  & 9  & \text{generically reduced elementary component} \\
 4 & (1,5,4)  & 10 & \text{smoothable} \\
 5 & (1,5,5)  & 11 & \text{generically reduced elementary component} \\
 6\le r\le 14 & (1,5,r) & 6+r & \text{smoothable by Theorem~\ref{thm:main-tail}} \\
 15 & (1,5,15) & 21 & \text{smoothable by Lemma~\ref{lem:m3-smoothable}}
\end{array}
\]
Thus the one-step embedding dimension five picture is exactly as stated in
Corollary~\ref{cor:classification}: the only elementary component cases are
$r=3$ and $r=5$, and all other one-step loci are contained in the smoothable
component.  The first five rows reproduce the previously known results: the
small smoothable cases $r=1,2$, the classical elementary case $r=3$, the
smoothable exceptional case $r=4$, and the elementary case $r=5$ identified in
recent work.  The new dominance certificates supply precisely the missing
smoothability assertions for the tail $6\le r\le 14$, while the endpoint
$r=15$ is handled separately.  Hence the table gives a complete answer to the
one-step Shafarevich gap question in embedding dimension five.

This classification matches the existing embedding dimension five results.
The rows $r=1,2$ agree with the small length smoothability theorem of
Cartwright--Erman--Velasco--Viray.  The rows $r=3,4$ agree with the low degree
component analysis of Ga\l{}\k{a}zka--Keneshlou--\v{S}ivic.  The row $r=5$
agrees with the elementary component result of
Giovenzana--Giovenzana--Graffeo--Lella.  It also does not contradict the
existing tangent space predictions.  Those predictions distinguish when the
one-step locus can be a generically reduced elementary component.  They do not,
by themselves, prove smoothability.  The new dominance certificates supply the
missing smoothability statement for the remaining tail.  In particular, the
borderline case $r=8$ and the range $r>8$ do not give new one-step
elementary components in embedding dimension five.  The corresponding one-step
loci lie in the smoothable component.

For $r\ne3,5$, the containment in the smoothable component is proper.  Indeed,
\[
  \dim V_{5,r}=5+r(15-r),
  \qquad
  \dim \Rsm^5_{6+r}=5(6+r),
\]
and these dimensions are unequal for every $r\ne5$.  In particular they are
unequal in every smoothable case $r\ne3,5$.  Hence the smoothable one-step
loci are not irreducible components of the Hilbert scheme.

\appendix

\section{The coordinate rank algorithm}\label{app:algorithm}

This appendix records the finite field linear algebra used in
Proposition~\ref{prop:rank-certificates}.  The ancillary Macaulay2 file
implements exactly this algorithm.

Work over $\F_{32003}$ and use the ordered basis
\[
\begin{split}
  \mathcal B=(&y_1^2,y_1y_2,y_1y_3,y_1y_4,y_1y_5,
  y_2^2,y_2y_3,y_2y_4,y_2y_5, \\
  &y_3^2,y_3y_4,y_3y_5,
  y_4^2,y_4y_5,y_5^2)
\end{split}
\]
of $\Ttwo$.  For $a=(a_1,\ldots,a_5)$, the coefficient vector of
\[
  q(a)=\sum_i a_i y_i^2-\left(\sum_i a_i y_i\right)^2
\]
has diagonal coefficients
\[
  a_i-a_i^2
\]
and off-diagonal coefficients
\[
  -2a_i a_j\qquad (i<j).
\]
The derivative in the $a_\ell$ direction is
\[
  \frac{\partial q(a)}{\partial a_\ell}
  =y_\ell^2-2\left(\sum_i a_i y_i\right)y_\ell.
\]
Thus its coefficient in the $y_\ell^2$ position is $1-2a_\ell$, its
coefficient in the other square positions is $0$, and its coefficient in the
$y_i y_\ell$ position is $-2a_i$ for $i\ne\ell$.

Given rows $a^{(1)},\ldots,a^{(r)}$, form the $r\times15$ matrix $U$ whose
rows are the vectors $q(a^{(j)})$.  The tangent space at
$U\subset\Ttwo$ is
\[
  \Hom(U,\Ttwo/U)\simeq (\Ttwo/U)^r.
\]
There are two families of infinitesimal directions.

\begin{enumerate}[label=(\arabic*)]
\item For the parameter $a^{(j)}_\ell$, the corresponding vector in
$\Ttwo^{\oplus r}$ is zero in every slot except the $j$-th slot, where it is
$\partial q(a^{(j)})/\partial a^{(j)}_\ell$.
\item For a matrix unit $E_{uv}\in\mathfrak{gl}_5$, first convert the
coefficient vector of $q(a^{(j)})$ to its symmetric matrix $M_j$.  With the
basis convention above, the off-diagonal coefficient of $y_i y_j$ is twice the
$(i,j)$ entry of $M_j$.  The $j$-th slot of the infinitesimal direction is the
coefficient vector of
\[
  E_{uv}M_j+M_jE_{uv}^{\mathsf T}.
\]
\end{enumerate}

Let $C$ be the matrix with these $5r+25$ direction vectors as columns.  Let
$G$ be the matrix whose columns are the $r$ rows of $U$ placed in each of the
$r$ slots of $\Ttwo^{\oplus r}$.  Then the differential rank is computed by
\[
  \rank[\,C\mid G\,]-\rank G.
\]
This avoids choosing quotient coordinates on $\Ttwo/U$ and is the form used in
the verifier.

\end{document}